\documentclass{amsart}
\usepackage{amssymb}
\usepackage{amsmath}
\usepackage{amsfonts}
\usepackage{pstricks}
\usepackage{pst-node}
\usepackage{graphicx}
\theoremstyle{plain}
\newtheorem{thm}{Theorem}[section]
\newtheorem{prop}[thm]{Proposition}
\newtheorem{lmm}[thm]{Lemma}

\newtheorem{prb}[thm]{Problem}
\theoremstyle{remark}

\def\II{\mathbb{I}}
\def\pmc#1{\setbox0=\hbox{#1}
    \kern-.1em\copy0\kern-\wd0
    \kern.1em\copy0\kern-\wd0}

\def\vep{\varepsilon}

\def\vp{\varphi}

\def\op{\operatorname}

\def\sm{\setminus}
\begin{document}

\bigskip

\title[On $2$-dimensional cell-like Peano continua]{On $2$-dimensional nonaspherical cell-like Peano continua: A simplified approach}


\bigskip

\bigskip

\author[K.~Eda]{Katsuya Eda} \address{School of Science and Engineering, Waseda University, Tokyo 169-8555, Japan} \email{eda@logic.info.waseda.ac.jp}
\author[U.~H.~Karimov]{Umed H. Karimov} \address{Institute of Mathematics, Academy of Sciences of Tajikistan, Ul. Ainy $299^A$, Dushanbe 734063, Tajikistan} \email{umedkarimov@gmail.com}
\author[D.~Repov\v{s}]{Du\v san Repov\v s } \address{Faculty of Mathematics and Physics, and Faculty of Education, University of Ljubljana, P.O.Box 2964, Ljubljana 1001, Slovenia} \email{dusan.repovs@guest.arnes.si}

\subjclass[2010]{Primary 54F15, 55N15; Secondary 54G20, 57M05}

\keywords{Noncontractible compactum, weak homotopy equivalence, trivial shape, Peano continuum, Snake cone, Alternating cone, asphericity, cell-like space, Topologist sine curve}

\begin{abstract}
We construct a functor $AC(-,-)$ from the category of path connected 
spaces $X$ with a base point $x$ to the category of simply
connected spaces. The following are the main results of the paper:
(i) If $X$ is a Peano continuum then
$AC(X,x)$ is a cell-like Peano continuum;
(ii) If $X$ is
$n-$dimensional
then $AC(X, x)$ is $(n+1)-$dimensional; and
(iii) For a path connected space $X$, $\pi_1(X,x)$ is trivial if and only
 if $\pi _2(AC(X, x))$ is trivial. 
As a corollary,  $AC(S^1, x)$ is a 2-dimensional nonaspherical cell-like
 Peano continuum.
\end{abstract}

\date{\today}

\maketitle

\section{Introduction}

{\it Peano} continua are those continua (i.e. connected metric compacta)
which are also locally connected. This classical topics of topology has
received renewed attention in recent years, in particular the
constructions of some surprising examples which provide
answers to some
very
interesting problems and conjectures.

 In our investigations of Peano continua we constructed in \cite{EKR FM}
 a new functor  from the category of all path connected
spaces $X$ with a base point $x$ and continuous mappings to the category
of all simply connected spaces,
which we named the
{\it Snake cone} and denoted by
$SC(-,-)$. 
When $X$ is a Peano continuum, $SC(X,x)$ is also a Peano continuum. 
We proved that in the case when $X=S^1$,
the space $SC(S^1, x)$ is a noncontractible simply
connected cell-like $2$-dimensional Peano continuum.

In a subsequent paper \cite{EKR TA} we showed that the space
$SC(S^1, x)$ is 
even nonaspherical, i.e.
$\pi_2(SC(S^1, x)) \neq 0$. 
Moreover, we showed that for every path connected space $X$, 
$\pi _1(X,x)$ is trivial if and only if $\pi _2(SC(X,x))$ is trivial
(cf. \cite{EKR TA} and
\cite{EKR Glasnik}). 

In the present paper we construct yet 
another new useful functor, called the
{\sl Alternating cone}, and we denote it by $AC(-,-)$.
We show that
while
$AC(X,x)$ shares all the 
properties of the space $SC(X,x)$ listed above, 
their verification is much easier. 
We also show that there exists a Peano continuum $(X,x)$ such that
the spaces $AC(X,x)$ and $SC(X,x)$ are not homotopy
equivalent. 
As an example how much easier the verifications are for the new functor,
we consider singular homology $H_{*}(AC(X,x))$ (cf. Theorem~\ref{thm:3}). 
\date{\today}

\section{Preliminaries}
We start by fixing some terminology and notations which will be
used in the sequel. All undefined terms can be found in \cite{S}
or
\cite{W}.

For any two points $A$ and $B$ of the plane $\mathbb{R}^2$, we
shall denote the {\it linear segment} connecting these points by
$[A,B]$. Next, by the {\it open} (resp. {\it half-open}) {\it interval} 
we shall mean the
linear segment without its end points (resp. without one of its end
points), and we shall denote it
by $(A, B)$ 
(resp.$[A, B)$ or $(A,B]$), as
usual. The unit segment of the real line $\mathbb{R}^1$, with the natural
topology, will be denoted by $\mathbb{I}.$ The point of the coordinate
plane $R^2$ with coordinates $a$ and $b$ will be denoted by $(a; b).$

Consider the points
$A = (0; 0), B = (0; 1), A_n = (1/n;
0), B_n = (1/n; 1)$ 
and let $L = [A, B],$ $L_{2n-1} =
[A_n, B_n],$ $L_{2n} = [B_n, A_{n+1}]$ be
segments in the
plane $\mathbb{R}^2$, for $n \in \mathbb{N} = \{1, 2, 3, \dots \}.$
The piecewise linear {\it Topologist sine curve}
$\mathcal{T}$ is a subspace of the plane $\mathbb{R}^2$ defined as the
union of $L_n$ and $L.$
For any topological space $X$ the cone $C(X)$ over $X$ is defined
as the quotient space $C(X) = (X\times \mathbb{I})/(X\times \{ 1\})$.

Suppose that $q_i: Z_i \to \mathbb{I}, \  i = 1, 2$, are  mappings of
topological spaces. Then the mapping $f:Z_1\to Z_2$ is said to be
{\it flat} with respect to the mappings $q_1$ and $q_2$ if for every two
points $a, b\in Z_1$ with $q_1(a) = q_1(b)$, the equality
$q_2(f(a)) = q_2(f(b))$ holds.
Homotopy $H:Z_1\times \mathbb{I} \to Z_2$ is said to be {\it flat} if
for each $t\in \mathbb{I}$ the mapping $H(-, t):Z_1 \to Z_2$ is a flat
mapping with respect to the mappings $q_1$ and $q_2$.

We recall the notion of 
the free $\sigma$-product of groups and a lemma from \cite{E:free}
(cf. the proof of \cite[Theorem 3.1]{EKR TA}). We only use it in the case 
when the index set is countable and we introduce a 
restricted version 
in the sequel. 
Let $(X_n,x_n)$ be pointed spaces for $n\in \mathbb{N}$ 
such that $X_m \cap X_n = \emptyset $, for $m \neq n.$ 
The underlying set of the
pointed space 
$(\widetilde{\bigvee}_{n\in \mathbb{N}}(X_n,x_n),x^*)$ 
is the union of all $X_n$'s obtained by identifying all $x_n$'s to a point
$x^*$ and the topology is defined by specifying the neighborhood base as 
follows (cf. \cite{BM}): 

\begin{itemize}
\item[(1)] If $x \in X_n \setminus \{x_n\}$, then the neighborhood base
of $x$ in $\widetilde{\bigvee}_{n\in \mathbb{N}}(X_n,x_n)$ is the one of $X_n$; and  
\item[(2)] The point $x^*$ has a neighborhood base, each element of which is
of the form:  $\bigcup _{n\ge m}(X_n,x_n) 
	   \cup \bigcup _{n < m}U_n$ 
where $m\in \mathbb{N}$ and each $U_n$
is an open neighborhood of $x_n$ in $X_n$ for $n < m$. 
\end{itemize}

\begin{lmm}\label{lmm:first}\cite[Theorem A.1]{E:free} 
Let $X_n$ be locally simply-connected and first countable at $x_n$ for
 each $n$. Then 
 $$\pi _1(\widetilde{\bigvee}_{n\in \mathbb{N}}(X_n,x_n),x^*)
\cong \pmc{$\times$}\, \, \; _{n\in \mathbb{N}}\pi _1(X_n,x_n).$$  
\end{lmm}

Let $G$ be a group. A {\it commutator}  
is an element $[g_1,g_2]=g_1g_2g_1^{-1}g_2^{-1}$ 
where $g_1,g_2\in G$. The normal subgroup
$G^{\prime}$ of $G$ which is generated by all commutators is
called the {\it commutator subgroup} of $G$.

Recall that the {\it commutator length} $cl(g)$ of $g\in G^{\prime}$ is
defined as the minimal number of the commutators of  $G$ whose
product is equal to $g$ (cf. \cite{EKR cl}). 
We note that $cl(e) = 0$ for the identity element $e$ of the group
$G$. If $\varphi :G \to H$ is a homomorphism of a group $G$ to a group
$H$, then for every $g\in G^{\prime},$
\begin{equation}\label{equation:cl(phi)}
cl(\varphi (g)) \leq cl(g).
\end{equation}
For every path connected space X, the fundamental group $\pi_1(X,
x)$ is independent of the choice of the base point $x$
and hence we can
simply write $\pi _1(X)$. Finally, recall the isomorphism:

\begin{equation}\label{equation:H_1}
H_1(X)\cong \pi_1(X)/\pi_1^{\prime}(X).
\end{equation}

\section{Constructions and main results}

The Snake cone functor $SC(-, -)$ is defined as follows. For every compact
space $X$ with the base point $x$, the space $SC(X, x)$ is the
quotient space of the topological sum $(\mathcal{T}\times X)
\bigsqcup \mathbb{I}^2$ via the identification of the points
$(t,x)\in \mathcal{T}\times X$ with $t\in \mathcal{T} \sm L \subset
\mathbb{I}^2$ and the identification of each set $\{t\}\times
X$ with the
point $t$, for every $t\in L$ (cf.
\cite{EKR FM}).

\begin{figure}[!htb]
\centering
\includegraphics[scale=1.0]{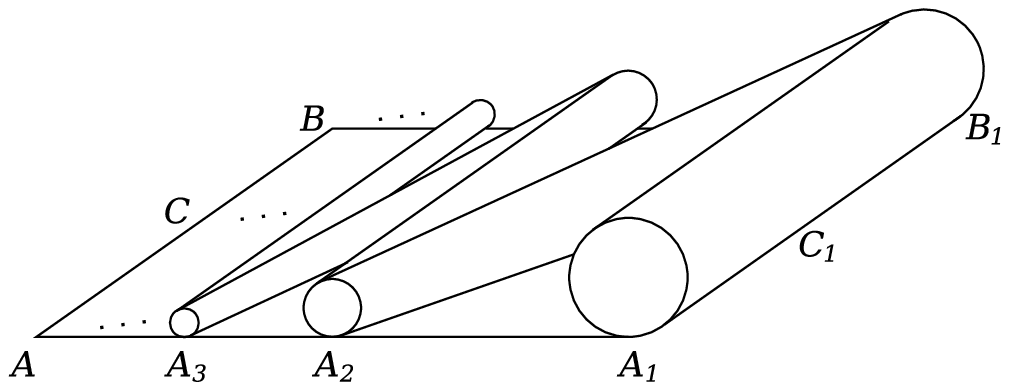}
\caption{$SC(S^1)$} \label{fig:fig-1}
\end{figure}

Let us now define the Alternating cone functor $AC(-, -).$ 
Let $(X,x)$ be a pointed space and let $(X_n,x_n)$ be its copies. 
Connect $x_n$ and $x_{n+1}$ by an arc for each $n\in\mathbb{N}$ 
and add a new point so that 
a neighborhood of the point contains almost all $X_n$'s and  
almost all attached arcs. Let $Y$ be this space. 

Let $AC(X, x)$ be the quotient space of $Y\times \mathbb{I}$ in which the
sets $\{ 0\}\times X_n$ are identified with $(0,x_n)$ for all odd $n$ and 
the sets $\{ 1\}\times X_n$ are identified with $(1,x_n)$ for all even $n$. 

Since $x_n$ and the attached arcs converge to the new point in $Y$, 
$AC(X, x)$ can be illustrated as in Figure 2. 
The identifited points will be denoted by $A_{2n} = (1/2n;
0)$ and $B_{2n-1} =(1/(2n-1); 1)$.  
Let $p:AC(X, x) \to \mathbb{I}^2$ be the natural projection and 
define
$p(u) = (p_1(u); p_2(u))$. 
We let $C_n = (1/n; 1/2)$,
for $n\in \mathbb{N}$ and $C=(0;1/2)$. 
\begin{figure}[!htb]
\centering
\includegraphics[scale=1.0]{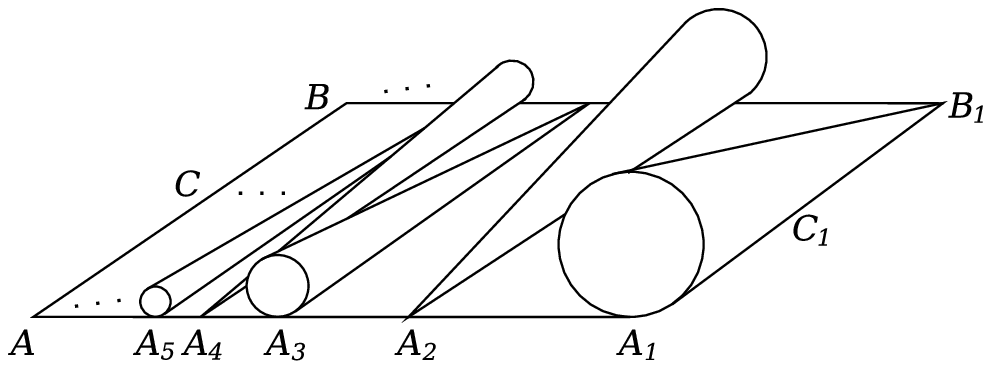}
\caption{$AC(S^1)$} \label{fig:fig-2}
\end{figure}

\begin{prop}\label{prop:basepoint} 
Let $X$ be a path connected space and $x,y\in X$. Then 
$SC(X,x)$ and $SC(X,y)$ are homotopy equivalent. Similarly, 
$AC(X,x)$ and $AC(X,y)$ are also homotopy equivalent. 
\end{prop}
\begin{proof}
We shall only prove this for $SC(X,x)$ and $SC(X,y)$, since the
proof for $AC(X,x)$ and $AC(X,y)$
is similar. 
Let $X$ be a path connected space with $x\in X$, and $X^*$  the 
space defined in Theorem~\ref{thm:3}. Also, let $Y^*$ be the space
 obtained by replacing $x$ with $y$. 
By a deformation on $\mathbb{I}^2$, $SC(X,x)$ is homotopy
 equivalent to $SC(X^*,0)$. Hence 
it suffices to establish
a homotopy equivalence
 between $SC(X^*,0)$ and $SC(Y^*,0)$. 
 
 Let $f:\mathbb{I}\to X$ be a path
 such that $f(0) = x$ and $f(1) = y$. 
Define $\vp: X^* \to Y^*$ and $\psi: Y^* \to X^*$ by: 
\[
  \vp (u) =
\begin{cases}
 u &  \quad \mbox{ for }u\in X  \\
 2s &\quad \mbox{ for }s\in [0,1/2] \\
 f(2-2s)&\quad \mbox{ for }s\in [1/2, 1]
\end{cases}
\]
\[
 \psi (u) = 
\begin{cases}
u & \quad \mbox{ for }u\in X  \\
2s & \quad \mbox{ for }s\in [0,1/2] \\
f(2s -1) &\quad \mbox{ for }s\in [1/2, 1].
\end{cases}
\]
Then it is easy to see that $\vp$ and $\psi$ induce a homotopy
 equivalence between $X^*$ and $Y^*$ which fixes $0$. 
This homotopy equivalence then
induces a homotopy equivalence between
 $SC(X^*,0)$ and $SC(Y^*,0)$. 
\end{proof}

We are mainly interested in the homotopy types of $SC(X,x)$ and $AC(X,x)$
for the class of path connected spaces $X$.
Since by Proposition~\ref{prop:basepoint} the
choice of the base point $x$ does not affect the homotopy types of 
$SC(X,x)$ and $AC(X,x)$, we can
simply write $SC(X)$ or $AC(X)$ when the
choice of the base point does not matter. 

\begin{thm}\label{thm:1}
If  $X$ is a continuum then $AC(X)$ is a cell-like space. If ${\rm
dim}\ X = n$ then ${\rm dim}\ AC(X) = n+1.$ If $X$ is a Peano
continuum, then $AC(X)$ is also a Peano continuum. For any
continuum $X$, the space $AC(X)$ is  simply connected.
\end{thm}

\begin{thm}\label{thm:2}
If $X$ is a noncontractible compact space, then $AC(X)$ is also a
noncontractible compact space.
\end{thm}

\begin{thm}\label{thm:3}
Let $X$ be a path connected space with $x\in X$ and $X^*$  the 
quotient space of the disjoint sum of the unit interval $\mathbb{I}$ and $X$,
 by identifying $1$ of $\mathbb{I}$ and $x$. Let $(X_i,x_i)$ be copies
 of $(X^*, 0)$. 

It then follows that
for any natural number $n$, we have the isomorphism:
$$H_{n+1}(AC(X))\cong H_n(C(\tilde{\bigvee}_{i\in \mathbb{N}}(X_{2i-1},x_{2i-1}))\vee
C(\tilde{\bigvee}_{i\in \mathbb{N}}(X_{2i},x_{2i}))),$$
where the attaching
 points of
 the two cones are the identification points $x_{2i-1}$ and
 $x_{2i}$. 
\end{thm}

\begin{thm}\label{thm:4}
Let $X$ be a path connected space. Then 
$\pi _1(X)$ is trivial if and only if $\pi_2(AC(X))$ is trivial. 
\end{thm}

It is easy to see that if $\{a, b\}$ is a two-point space with the
discrete topology then
the space $AC(\{a, b\})$ is path connected,
whereas
the space $SC(\{a, b\})$
is not.  The following theorem shows their difference from another
perspective: 

\begin{thm}\label{thm:5}
There exists a Peano continuum $X$ such that $AC(X)$
is not homotopy equivalent to $SC(X)$.
\end{thm}

\section{Proofs of Theorems~\ref{thm:1} and \ref{thm:2}}

In this section we shall prove Theorems~\ref{thm:1} and \ref{thm:2}. 
The argument is similar to the one which was used for proving analogous 
statements concerning 
$SC(X)$.  
\medskip

\noindent
{\it Proof}~ of Theorem~\ref{thm:1}. 
Since $X$ is a compactum, the space $AC(X)$ is the inverse limit of an
inverse sequence of contractible
compact metrizable  spaces.
Therefore $AC(X)$ is a compact metrizable space of trivial shape, i.e.  cell-like.

If ${\rm dim}\ X = n$ and $X$ is compact and 
metrizable then the
dimension of its
one-point compactification $Y$ is equal to $n$  and thus
${\rm dim}\ (Y\times \mathbb{I}) = n+1.$ Therefore 
${\rm dim}\ AC(X) = n+1.$
If $X$ is a
Peano continuum then  $AC(X)$ is compact, metrizable and locally path connected, hence also a
Peano continuum.

Let us prove that $AC(X)$ is a simply connected space. Let $U =
p_2^{-1}([0, 1))$ and $V = p_2^{-1}((0, 1])$ (cf. Figure~\ref{fig:fig-1}). 
The sets
$U$ and $V$ are homotopy equivalent
to $p_2^{-1}(\{ 0\})$ and $p_2^{-1}(\{ 1\})$, respectively. Their
intersection
$U\cap
V$ is homotopy equivalent to  $p_2^{-1}(\{ 1/2\})$. Obviously,
the mapping
$$\pi_1(p_2^{-1}(\{ 1/2\})) \rightarrow \pi_1(p_2^{-1}(\{ 0\}))\ast \pi_1(p_2^{-1}(\{ 1\}))$$
is surjective and the
intersection $U\cap V$ is path connected.
Therefore by the van Kampen theorem, the union $U\cup V =
AC(X)$ is  simply connected. \qed

For a
noncontractible space $X$, the noncontractibility of $AC(X)$ can be 
proved similarly as for $SC(X)$ (cf. \cite[Theorem
1.2]{EKR FM}). Hence we only indicate the difference and we state the necessary lemmas. 

Let $M_y = p_2^{-1}(\{ y\})$ and $AC_n(X) = p^{-1}([A_{2n}, B_{2n}]\cup
[A_{2n+1},B_{2n+1}])$ be the subspaces of $AC(X)$. 
The following lemma can be proved  analogously as \cite[Lemma
4.2]{EKR FM}. 
\begin{lmm}\label{lmm:flat}
Let $n\in \mathbb{N}$ and let
$H:AC_n(X)\times\mathbb{I}\to AC(X)$ be a
 mapping such that for every $y\in \mathbb{I}$ and $t\in \mathbb{I}$,
 the closure of the set
$p_2(H(M_y\cap AC_n(X), t))$ does not contain both points
$0$ and $1$, and such that both mappings $H(-,0)$ and $H(-,1)$ are flat. 
Then there exists a flat homotopy from $H(-,0)$ to $H(-,1)$.
\end{lmm}

For $s\in (0,1)$ and $t\in \mathbb{I}$, we define a property
$P(s,t)$ of the
flat homotopy $H$ as follows: 
\begin{quote}
$H(M_s\cap AC_n(X), t)\subseteq p_2^{-1}((0,1))$ 
and the restriction of $H(-,t)$ to $M_s\cap AC_n(X)$
is homotopic to the identity mapping on $M_s\cap AC_n(X)$ in 
$p_2^{-1}((0,1))$.
\end{quote}
We remark that by the flatness of $H$, if 
$H(M_s\cap AC_n(X), t)\subseteq p_2^{-1}((0,1))$,  
then there is a neighborhood $U$ of $(s;t)$ such that 
$H(M_{s'}\cap AC_n(X), t')\subseteq p_2^{-1}((0,1))$ for
any $(s';t')\in U$. 
The following lemma can be proved similarly as \cite[Lemma 4.3]{EKR FM}. 
\begin{lmm}\label{lmm:crucial}
Let $X$ be a noncontractible space and $H:AC_n(X)\times\mathbb{I}
\to AC(X)$ a flat homotopy. If $ H(M_0\cap AC_n(X),t_0)
\subseteq p_2^{-1}((0,1))$, then
 there exists a neighborhood $U$ of $(0;t_0)$ such that $H$ does not
 satisfy $P(s,t)$ for
 any $(s;t)\in U$ with $s>0$. An analogous statement holds for 
$H(M_1\cap AC_n(X),t_0) \subseteq p_2^{-1}((0,1))$. 
\end{lmm}
By Lemmas~\ref{lmm:flat} and \ref{lmm:crucial}, we can use
the proof of \cite[Lemma 4.4]{EKR FM} to verify also the next lemma. 

\begin{lmm}\label{lmm:noncontr}
Let $X$ be a noncontractible space.
If $H:AC_n(X)\times \mathbb{I}\to AC(X)$ is a flat homotopy such
that $H(u, 0) = u$ for every $u\in AC_n(X)$, then $H(-,1)$ is not a
 constant mapping.
\end{lmm}

We can now
regard the proof of \cite[Theorem 1.2]{EKR FM} also as a proof of
Theorem~\ref{thm:2}. We shall only retrace its line of argument. 
\medskip

\noindent
{\it Proof}~ of Theorem~\ref{thm:2}. Suppose that the space $AC(X)$ is contractible. Then there exists
a homotopy $H:AC(X)\times T\to AC(X)$ connecting the identity
mapping with the constant one. By 
compactness of the time interval $\mathbb{I}$, for every $a = (0;y) \in \{ 0\}
\times \mathbb{I}$, there exists $\vep _0>0$ such that the diameters of
$p_2\circ H(O_{\vep _0}(a),t)$ are less than $1$ for all $t\in
\mathbb{I}$. Hence, by compactness of $[A,B]$, 
there exists $\vep _1> 0$ such that the diameters 
$p_2\circ H(O_{\vep _1}(a),t)$ are less than $1$ for all $a = (0;y)\in \{ 0\}
\times \mathbb{I}$ and all $t\in \mathbb{I}$. 
Let $n$ be a number such that $1/n < \vep _1$. By Lemma \ref{lmm:flat}
we may then assume that $H|_{AC_n(X)\times \mathbb{I}}$ is a flat
contraction. However, this contradicts Lemma \ref{lmm:noncontr}.
\qed

\section{Proofs of Theorems~\ref{thm:3} and \ref{thm:4}}
In this section we prove Theorems~\ref{thm:3} and \ref{thm:4}. 
Here we shall
see a very different feature of $AC(X)$ in comparison with
$SC(X)$.  

Let $U = p_2^{-1}([0,1))$ and $V = p_2^{-1}((0, 1])$.  
The spaces $U, V$ and $U\cap V$ are homotopy equivalent to
$p_2^{-1}(\{ 0\}), p_2^{-1}(\{ 1 \})$ and $p_2^{-1}(\{ 1/2\} )$, respectively. 
Let $i_U$ and $i_V$ be the inclusion maps from $U\cap V$ to $U$ and $V$,
respectively. 
Consider
the following part of the Mayer-Vietoris exact sequence for the
singular homology:
\begin{eqnarray*}
&&H_{n+1}(U)\oplus H_{n+1}(U)\longrightarrow H_{n+1}(AC(X))\longrightarrow H_n(U\cap V)\stackrel{(i_{U*}, i_{V*})}
\longrightarrow \\
&&\qquad \qquad H_n(U)\oplus H_n(V)\longrightarrow 
H_n(AC(X)).
\end{eqnarray*}
Since the inclusion maps from $U$ to $AC(X)$ and from $V$ to $AC(X)$ are inessential, we have the following exact sequence: 
\[
0 \longrightarrow H_{n+1}(AC(X))\longrightarrow H_n(U\cap V)
\stackrel{(i_{U*}, i_{V*})}\longrightarrow 
 H_n(U)\oplus H_n(V)\longrightarrow 0.
\]

We recall the following fact:
\begin{lmm}\label{lmm:cone}\cite[Theorem 1.3]{EK:aloha}
Let $CX$ be the cone over a space $X$ and $CX\vee CY$ 
the one-point union with two points of the base spaces 
$X$ and $Y$ being identified to a point. 
Then, the following holds for $n\ge 1:$  
\[
\op{H}_{n}(X\vee Y)
\cong  
\op{H}_{n}(X)\oplus\op{H}_{n}(Y)\oplus \op{H}_{n}(CX\vee CY).
\]
\end{lmm}
\medskip

\noindent
{\it Proof}~ of Theorem~\ref{thm:3}. 
Let $r_X: X\vee Y \to X$ and $r_Y: X\vee Y \to Y$ be the retractions 
from
Lemma~\ref{lmm:cone}. 
Then $\op{H}_{n}(CX\vee CY)$ is the kernel of the homomorphism 
$((r_X)_* ,(r_Y)_*) : \op{H}_{n}(X\vee Y)\to
\op{H}_{n}(X)\oplus\op{H}_{n}(Y)$. 

\begin{equation*}
\begin{tabular}{crrrc}
$V$ & $\cong$ & $p_2^{-1}(\{ 1\})$ & $=$ & $p_2^{-1}(\{ 1\})$ \\ 
\multicolumn{1}{r}{$\uparrow i_V$} & {}  &{} &  {} & \multicolumn{1}{c}{$\uparrow r_B$}\\ 
$U\cap V$ & $\cong$ & $p_2^{-1}(\{ 1/2\})$ & $\cong$ &
 $(p_2^{-1}(\{ 0\}), A) \vee (p_2^{-1}(\{ 1\}),B)$ \\
\multicolumn{1}{r}{$\downarrow i_U$} & {} &{} & {} &
 \multicolumn{1}{c}{$\downarrow r_A$} \\ 
$U$ & $\cong$ & $p_2^{-1}(\{ 0\})$ & =  & $p_2^{-1}(\{ 0\})$ 
\end{tabular}
\end{equation*}
Since the diagram above
is homotopy commutative, 
$H_{n+1}(AC(X))$ is isomorphic to $H_n(C(p_2^{-1}(\{ 0\})) \vee
C(p_2^{-1}(\{ 1\}))$.

For each $n$, let $X_n$ be a copy of the one-point union of $X$ and 
the unit interval 
$\mathbb{I}$, where the attaching point is $0\in\mathbb{I}$.
Let $x_n$ be a copy of $1\in \mathbb{I}$. 
We regard $X_n$ as
$[C,C_n]\cup p^{-1}(\{ C_n\})$. 
Then $p_2^{-1}(\{ 0\})$ is homotopy equivalent to 
$\widetilde{\bigvee}_{n\in \mathbb{N}}(X_{2n-1},x_{2n-1})$, 
 $p_2^{-1}(\{ 1\})$ is homotopy equivalent to 
$\widetilde{\bigvee}_{n\in \mathbb{N}}(X_{2n},x_{2n})$ and  
$p_2^{-1}(\{ 1/2\})$ is homotopy equivalent to 
$\widetilde{\bigvee}_{n\in \mathbb{N}}(X_n,x_n)$. 
This verifies the assertion. 
\qed
\medskip

\noindent
{\it Proof}~ of Theorem~\ref{thm:4}. 
By the result preceding  the proof of Theorem~\ref{thm:3}, we have the exact sequence: 
\[
0 \longrightarrow H_2(AC(X))\longrightarrow H_1(U\cap V)
\stackrel{(i_{U*}, i_{V*})}\longrightarrow 
 H_1(U)\oplus H_1(V)\longrightarrow 0.
\]
Suppose that $\pi _1(X)$ is trivial. Since $\pi _1(U\cap V)$ is isomorphic 
to 
$\pmc{$\times$}\, \, \; _{n\in\mathbb{N}}\pi _1(X_n,x_n)$, $\pi _1(U\cap V)$ is also trivial. Hence $H_1(U\cap V)$ is trivial, which implies that 
$H_2(AC(X))$ is trivial,
by the above exact sequence. 

On the contrary, suppose that $\pi _1(X)$ is nontrivial. Choose a
non-trivial element $a$ of $\pi _1(X)$ and let $a_n\in \pi _1(X_n)$
be copies of $a$.   
Then there exists an element $a^*$ of the group $\pi_1(U\cap
V)$ such that its image under the induced homomorphism from retraction
of the space $U\cap V$ to the $[C, C_{2n}]\cup p^{-1}([C_{2n}, C_1])$ is
$$[a_1, a_2][a_3, a_4]\cdots [a_{2n-1}, a_{2n}].$$ 
By \cite[Theorem 1]{Griffiths:commutator},
the commutator length of this
element is  $n$. Since the number
$n$ is arbitrary, $a^*$ does not belong to the
commutator subgroup of $\pi _1(U\cap V)$. 
Therefore the homology class $[a^*] \in H_1(U\cap V)$ is nontrivial
(\ref{equation:H_1}). 
However, its images $i_{U*}([a^*])$ and $i_{V*}([a^*])$ in $H_1(U)$ and
$H_1(V)$, respectively, are obviously trivial. By the exactness of
the Mayer-Vietoris sequence, the group $H_2(AC(X))$ is nontrivial.
Since by Theorem~\ref{thm:1},
$\pi_1(AC(X))$ is trivial,
it follows
by the Hurewicz Theorem that $\pi_2(AC(X))$ is isomorphic to
$H_2(AC(X))$ and hence is also nontrivial.\qed
\medskip

\noindent
{\it Proof}~ of Theorem~\ref{thm:5}. 
Let $\mathbb{H}$ be the well-known Hawaiian earring, which is obviously a
Peano continuum.
Let us show that $SC(\mathbb{H})$ is not homotopy equivalent to
$AC(\mathbb{H}).$ Consider the embedding $\varphi: \mathbb{H} \to
SC(\mathbb{H})$, defined by $\varphi(x) = (A_1, x) \in
SC(\mathbb{H})$ in Figure \ref{fig:fig-1}. In the
sequel we shall
identify $\mathbb{H}$ with its image $\varphi (\mathbb{H}).$

We show that the mapping $\varphi$ is not homotopy equivalent to
the constant one. To show this
by contradiction, suppose that this is
not the case and that there exists
a homotopy $H:X\times I \to
SC(\mathbb{H})$ such that $H(x, 0) = \varphi (x)$ and $H(-, 1)$ is
a constant mapping to some point $\ast$.

Since $SC(\mathbb{H})$ is a path connected space, we may assume
without loss of
generality, that $\ast \in [A, B] \subset
\mathcal{T}.$ Let $a = \inf \{t : H(x^*, t)\in [A, B]\}.$ The
image $H(x^*\times [0, a])$ is a Peano continuum such that
$H(x^*,0) = (A_1, x^*)\in \mathcal{T}\subset SC(X),\ H(x^*, 1)\in
[A, B]\subset \mathcal{T}\subset SC(X).$ Since $[A,B]$ is a
path
connected component of the space $\mathcal{T}$, the set $p\circ
H(x^*\times [0, a])$ is not a subset of $\mathcal{T}$ and there
exists $b\in [0, a]$ for which $p\circ H(x^*, b) \notin
\mathcal{T}$ and consequently $H(x^*,b)\in \II ^2\sm \mathcal{T}$.

Since $b<a$, we have $p\circ H(x^*,[0,b])\cap [A,B] = \emptyset$
and so  $H(x^*,[0,b])\cap [A,B] = \emptyset$. There exists a
neighborhood $D$ of $H(x^*,b)$ which is a disk and is
disjoint from
$\mathcal{T}.$ We have a neighborhood $W$ of $x^*$ for which
$H(W\times [0, b])\subset (SC(X)\setminus [A, B])$ and $H(W\times
b)\subset D.$ Since the disk $D$ is contractible and the space
$\mathbb{H}$ is a retract of $SC(X)\setminus [A,B]$, we get
a
homotopy of $W$ to $\mathbb{H}$ connecting the embedding with a
constant mapping. However, the embedding of $W$ to $\mathbb{H}$ is
essential, which is a contradiction. Now we conclude that $\varphi$ is
an essential mapping.

Next we show that every mapping $\psi:\mathbb{H}\to
AC(\mathbb{H})$ is inessential. First, observe that the embeddings
of $U$ and $V$ to $AC(\mathbb{H})$ are inessential. For instance,
for $U$, we have a deformation retraction of $U$ to $p_y^{-1}(A)$
which can be contracted in $F(\mathbb{H})$ to a
point. We have $\psi
(x^*)\in U$ or $\psi (x^*)\in V$ and we only deal with the case
$\psi (x^*)\in U$.

There exists a neighborhood $\mathcal{O}$ of
$x^*$ such that $\psi (\mathcal{O})\subset U$. All but finitely
many circles of $\mathbb{H}$ are contained in $\mathcal{O}$. Since
$AC(X)$ is simply connected by Theorem~\ref{thm:1}, the
restrictions of $\psi$ to the finitely many circles are inessential
mappings. Hence, as a total, $\psi$ is an inessential map.

Suppose now that the spaces $SC(X)$ and $AC(X)$ are homotopy
equivalent. Then there exist mappings $f:SC(X)\to AC(X)$ and
$g:AC(X)\to SC(X)$ such that the composition $gf$ is homotopic to
the identity mapping. Thus $\varphi \simeq gf\varphi$. However,
$gf\varphi$ is an inessential mapping, because $f\varphi$ is
inessential and $\varphi$ was shown to be an essential mapping.
Contradiction. \qed \\

In view of Theorems~\ref{thm:1},~\ref{thm:2},
and~\ref{thm:3}, it is natural to ask
the following question
(cf. \cite{EKR TA}):

\begin{prb}\label{1}
Does there exist a finite-dimensional noncontractible Peano \break
continuum all homotopy groups of which are trivial?
\end{prb}






\section{Acknowledgements}
This research was supported by the Slovenian Research Agency
grants P1-0292-0101, J1-9643-0101, J1-2057-0101 and J1-4144-0101.
The first author was also
supported by the
Grant-in-Aid for Scientific Research (C) of Japan
No. 20540097 and 23540110.
We thank the referee for comments and suggestions. 

\providecommand{\bysame}{\leavevmode\hbox to3em{\hrulefill}\thinspace}

\end{document}